\documentclass[11pt, letterpaper]{amsart}

\usepackage[T1]{fontenc}
\usepackage{microtype}
\usepackage{xcolor}
\usepackage[osf]{newpxtext}
\usepackage[euler-digits]{eulervm}
\usepackage{inconsolata}

\usepackage{hyperref}
\usepackage{amsmath,amssymb,enumitem}

\usepackage{tikz,float,caption}
\usetikzlibrary{arrows.meta,calc,decorations.markings,patterns,cd,patterns.meta}

\usepackage[bb=ams,scr=euler]{mathalpha}

\usepackage{setspace}
\usepackage{thmtools}
\declaretheoremstyle[headfont=\normalsize\normalfont\bfseries,notefont=\mdseries, notebraces={(}{)},bodyfont=\normalfont,postheadspace=0.5em]{basicstyle}
\declaretheoremstyle[headfont=\normalsize\normalfont\bfseries,notefont=\mdseries,
notebraces={(}{)},bodyfont=\normalfont\itshape,postheadspace=0.5em]{italstyle}

\declaretheorem[style=italstyle,name=Theorem,numberwithin=section]{theorem}
\declaretheorem[style=italstyle,name=Corollary,sibling=theorem]{cor}

\declaretheorem[style=italstyle,name=Claim,sibling=theorem]{claim}

\declaretheorem[style=italstyle,name=Lemma,sibling=theorem]{lemma}

\makeatletter
\newcommand{\abs}[1]{|#1|}
\newcommand{\bd}{\partial}
\newcommand{\colim}{\mathop{\mathrm{colim}}}
\newcommand{\C}{\mathbb{C}}

\renewcommand{\d}{d}

\newcommand{\id}{\mathrm{id}}
\newcommand{\intprod}{\mathbin{{\tikz{\draw[line width=0.7pt](-0.1,0)--(0.1,0)--(0.1,0.2)}\hspace{0.5mm}}}}

\newcommand{\pd}[2]{\frac{\partial #1}{\partial #2}}

\newcommand{\R}{\mathbb{R}}
\def\@secnumfont{\bfseries}
\renewcommand\section{\@startsection{section}{1}{0pt}{-3.5ex \@plus -1ex \@minus -.2ex}{2.3ex \@plus.2ex}{\centering\itshape}}
\newcommand{\set}[1]{\left\{#1\right\}}
\renewcommand{\subsection}{\@startsection{subsection}{2}\z@{.5\linespacing\@plus.7\linespacing}{-.5em}{\normalfont\itshape}}
\renewcommand{\paragraph}{\@startsection{paragraph}{4}\z@{.5\linespacing\@plus.7\linespacing}{0em}{\normalfont\itshape}}

\newcommand{\ud}[2]{\frac{\mathrm{d} #1}{\mathrm{d} #2}}
\newcommand{\Z}{\mathbb{Z}}

\makeatother

\title{Selective Floer cohomology for contact vector fields}
\author{Dylan Cant}
\author{Igor Uljarevi{\'c}}
\date{\today}
\begin{document}
\begin{abstract}
  This paper associates a persistence module to a contact vector field $X$ on the ideal boundary of a Liouville manifold. The persistence module measures the dynamics of $X$ on the region $\Omega$ where $X$ is positively transverse to the contact distribution. The colimit of the persistence module depends only on the domain $\Omega$ and is a variant of the selective symplectic homology introduced by the second named author. As an application we prove existence of positive orbits for certain classes of contact vector fields. Another application of this invariant is that we recover the famous non-squeezing result of Eliashberg, Kim, and Polterovich.
\end{abstract}

\maketitle

\section{Introduction}
\label{sec:introduction}

The goal in this paper is to associate a persistence module of Floer cohomology groups to a contact vector field $X$ on the ideal contact boundary $Y$ of a Liouville manifold $W$. The Floer cohomology group is a special case of the cohomology groups considered in \cite{merry-uljarevic,djordjevic_uljarevic_zhang,cant-sh-barcode,cant-hedicke-kilgore} which are associated to arbitrary contact isotopies of $Y$.

We briefly sketch the construction. Given an auxiliary contact form $\alpha$ one can associate to $X$ the \emph{contact Hamiltonian} $h=\alpha(X)$. Let $\mu_{\delta}$ be a cut-off function of the form illustrated in Figure \ref{fig:mu}.

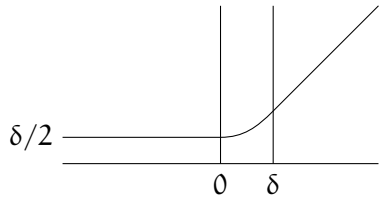
\begin{figure}[H]
  \centering
  \begin{tikzpicture}[scale=0.7]
    \draw (-3,0.5)node[left]{$\delta/2$} -- (0,0.5) to[out=0,in=225] (1,1)--(3,3);
    \draw (-3,0)--(3,0);
    \draw (0,0) node [below] {$0$}-- (0,3) (1,0)node [below]{$\delta$}--(1,3);
  \end{tikzpicture}
  \caption{The convex and positive cut-off function $\mu_{\delta}$ is required to satisfy $\mu_{\delta}(x)=x$ for $x\ge \delta$.}
  \label{fig:mu}
\end{figure}

The cut-off function $\mu_{\delta}(h)$ generates a new contact vector field $X_{\delta}^{\alpha}$ via the equation $\alpha(X_{\delta}^{\alpha})=\mu_{\delta}(h),$ and we let $HF(X_{\delta}^{\alpha};s)$ denote the Floer cohomology group associated to the contact isotopy obtained as the time-$s$ flow of $X_{\delta}^{\alpha}$. For $\delta$ sufficiently small, this new contact vector field ``selects'' the dynamics of $X$ only on its positive region (note that $X=X_{\delta}^{\alpha}$ on the region where $h\ge \delta$).

Our invariant, denoted $Q(X;s)$, is defined as an inverse limit (over continuation maps) as $\delta\to 0$ of the groups $HF(X_{\delta}^{\alpha};s)$; the precise construction is given in \S\ref{sec:defin-select-floer}. It is shown that $Q(X;s)$ is independent of the choice of contact form.

For our main application, it is useful to separate orbits by their free homotopy classes. Let us therefore fix $\kappa$ to be a collection of connected components in the free loop space of $Y$. Because Floer cohomology ultimately is defined inside the filling, we require $\kappa$ satisfies the following condition we call \emph{saturation}:
\begin{equation*}
  \kappa=i^{-1}(i(\kappa)),
\end{equation*}
where $i:Y\to W$ is the inclusion of the ideal boundary. We will refer to $\kappa$ as a \emph{saturated free homotopy class}. Given such a choice, we define a refined version of the invariant, denoted $Q(X;s;\kappa)$.

The vector spaces $Q(X;s;\kappa)$ are shown to form a persistence module with respect to $s\in [0,\infty)$; see \cite{polterovich_shelukhin_persistence_1,persistence_book} for earlier uses of persistence modules in symplectic topology. Assuming a dynamical condition called \emph{non-resonance}, we prove that the persistence module is supported on the lengths of positive orbits of $X$:
\begin{theorem}\label{theorem:pmod-barcode}
  Suppose that $X$ is non-resonant relative $\kappa$ and has no positive orbits with period in $[s_{0},s_{1}]\subset [0,\infty)$ in the saturated free homotopy class $\kappa$. Then the continuation map:
  \begin{equation*}
    Q(X;s_{0};\kappa)\to Q(X;s_{1};\kappa)
  \end{equation*}
  is an isomorphism.
\end{theorem}
Here a \emph{positive orbit} is an orbit of $X$ which is positively transverse to the contact distribution.

Non-resonance is a condition on the characteristic foliation of the hypersurface where $X$ is tangent to the contact distribution; the precise formulation is given in \S\ref{sec:non-resonance}. For the purposes of the introduction, let us note that:
\begin{enumerate}
\item non-resonance assumes the contact Hamiltonian $h=\alpha(X)$ cuts out its zero level set $\Sigma=\set{h=0}$ transversally,
\item assuming (1), non-resonance only depends on $\Sigma$,
\item if $\Sigma$ is \emph{convex} (in the sense of, e.g., \cite[\S2]{salamon-lecture-notes-convex-hypersurfaces}), then $X$ is non-resonant,
\item if the characteristic foliation of $\Sigma$ has no closed orbits in the class of $\kappa$, then $X$ is non-resonant.
\end{enumerate}
It follows from the construction and well-known results for Floer cohomology in Liouville manifolds that:
\begin{equation*}
  Q(X;0;\kappa)=\lim_{\delta\to 0}\mathrm{HF}(R^{\alpha}_{\delta t};\kappa)\simeq \left\{
    \begin{aligned}
      &H^{*}(W;\Z/2)\text{ if $\kappa$ contains constant loops},\\
      &0\text{ otherwise.}
    \end{aligned}
  \right.
\end{equation*}
At the other end of the persistence module we have an invariant closely related to the symplectic selective (co)homology introduced in \cite{uljarevic-ssh}. Introduce:
\begin{equation*}
  Q(X;\infty;\kappa):=\colim_{s\to \infty}Q(X;s;\kappa).
\end{equation*}
We will prove in \S\ref{sec:defin-select-floer} that there are continuation isomorphisms:
\begin{equation*}
  Q(X_{1};\infty;\kappa)\simeq Q(X_{2};\infty;\kappa)
\end{equation*}
provided that $\set{h_{1}\ge 0}=\set{h_{2}\ge 0}$ and each $h_{i}=\alpha(X_{i})$ cuts out its zero level set transversally. Thus, given any domain $\Omega\subset Y$ we define:
\begin{equation*}
  Q(\Omega;\kappa):=\lim_{X}Q(X;\infty;\kappa);
\end{equation*}
the limit is over vector fields $X$ whose positive region is $\Omega$ and which cut out the zero level set transversally, in which case we say $X$ is \emph{adapted to $\Omega$}. The limit is formal in the sense that the map from the limit to any representative is an isomorphism.

Our invariant $Q(\Omega;\kappa)$ is closely related to the selective symplectic homology introduced by the second named author in \cite{uljarevic-ssh}; however, the construction has certain mild differences, e.g., in this paper $\Omega$ is a compact domain rather than an open set. We discuss further comparison with \cite{uljarevic-ssh} in \S\ref{sec:comp-with-select}.

With these definitions settled, we can state an important result:
\begin{theorem}\label{theorem:main}
  Let $X$ be a non-resonant contact vector field adapted to $\Omega$. If $\kappa$ is non-trivial (contains only non-contractible orbits) and $Q(\Omega;\kappa)$ is non-zero, then $X$ has a positive orbit in the class $\kappa$. If $\kappa$ is trivial (contains contractible orbits) and the natural map $H^{*}(W;\Z/2)\to Q(\Omega;\kappa)$ is not an isomorphism, then $X$ has a positive orbit in the class of $\kappa$.
\end{theorem}
\begin{proof}
  This is a direct corollary of Theorem \ref{theorem:pmod-barcode} and the above definitions.
\end{proof}

This result should be thought of as the generalization of the famous result of \cite{viterbo_functors_and_computations_1} on Reeb vector fields to a larger class of vector fields.

In view of this theorem, it is worthwhile to compute the invariants $Q(\Omega;\kappa)$ in various settings, as the answer may imply the existence of positive orbits of any contact vector field $X$ adapted to $\Omega$.o

\subsection{The case of a contact Darboux ball}
\label{sec:case-darboux-ball}

Our first example where $Q(\Omega;\kappa)$ can be computed is when $\Omega$ is a contact Darboux ball, i.e.,
\begin{equation*}
  \Omega=\set{q^{2}+p^{2}+z^{2}\le 1}\subset \R\times \C^{n}.
\end{equation*}
with the contact form $\d z+\lambda$ where $\lambda$ is the radial Liouville form on $\C^{n}$. Our method of computation is to find a contact vector field $X$ adapted to $\Omega$ which has no closed positive orbits. A straightforward computation shows that:
\begin{equation*}
  h=1-z^{2}-p^{2}-q^{2}
\end{equation*}
is the contact Hamiltonian for a vector field $X=k\pd{}{z}+V,$ where $V$ is tangent to the level sets $\set{z=\text{const}}$ and:
\begin{enumerate}
\item $V\intprod \d p\wedge \d q=2p\d p+2q\d q-2 z \lambda$,
\item $k+\lambda(V)=h$.
\end{enumerate}
Insert the radial Liouville vector field $Z$ into both sides of the first equation to conclude:
\begin{equation*}
  -\lambda(V)=2p\d p(Z)+2q\d q(Z)\implies \lambda(V)\le 0.
\end{equation*}

It follows easily that $k\ge h$ and hence $X$ has no positive orbits, since $h>0$ holds on the region where $X$ is positive. Because the boundary of the ball is convex, $X$ is non-resonant, and hence Theorem \ref{theorem:main} implies that the natural morphism:
\begin{equation*}
  H^{*}(W;\Z/2)\to Q(\Omega;\kappa)
\end{equation*}
is an isomorphism when $\kappa$ is a trivial class (if $\kappa$ is nontrivial then $Q(\Omega;\kappa)$ vanishes because every loop in a ball is contractible). This isomorphism also follows from \cite[\S6]{uljarevic-ssh} and the comparison in \S\ref{sec:comp-with-select}.

\subsection{The prequantization of a symplectic Darboux ball}
\label{sec:case-prequantization-darboux-ball}

Consider $\R/\Z\times \C^{n}$, with coordinates $(\theta,z)$, and with the prequantization contact form:
\begin{equation}
  \label{eq:prequantization-form}
  \alpha=\d\theta+\lambda,
\end{equation}
where $\lambda$ is the radial Liouville form on $\C$. It is well-known that, for any $R>0$, $\R/\Z\times B(R)$ can be contactomorphically embedded into $\R/\Z\times B(\epsilon)$ for arbitrarily small $\epsilon>0$; see \cite{ekp,cant_nonsqueezing}. Here $B(a)$ denotes the ball of symplectic capacity $a$.

We will show in \S\ref{sec:prequant-sympl-ellip} that there is some Liouville manifold $W$ which contains an embedding of $\R/\Z\times B(R)$ into its ideal boundary $Y$ which has the following properties:

\begin{enumerate}
\item\label{item:knot-condition} the knot $K=\R/\Z\times \set{0}$ is not homotopic to any of its iterates in $W$; in particular, the saturation $\kappa$ of $K$ is non-trivial,
\item $Q(\R/\Z\times B(a);\kappa)\simeq \Z/2\oplus \Z/2$ if $a>1$,
\item\label{sec:item-3} the continuation map: $$Q(\R/\Z\times B(a);\kappa)\to Q(\R/\Z\times B(R);\kappa)$$ is an isomorphism if $1<a\le R$.
\end{enumerate}
Here we appeal to an additional structure we have yet to mention so far: for any inclusion of domains $\Omega_{1}\subset \Omega_{2}$ there is a map $Q(\Omega_{1};\kappa)\to Q(\Omega_{2};\kappa)$ induced by continuation maps. Moreover, the assignment $\Omega\mapsto Q(\Omega;\kappa)$ is functorial with respect to these continuation maps. This additional structure is explained in \S\ref{sec:select-sympl-cohom}.

This set-up allows us to conclude:
\begin{theorem}\label{theorem:existence}
  Let $\Omega\subset \R/\Z\times \C$ be any compact domain which contains the compact domain $\R/\Z\times B(1)$ in its interior. Suppose that $\bd\Omega$ is non-resonant. Then any contact vector field $X$ adapted to $\Omega$ has a positive orbit in the free homotopy class of $\R/\Z\times \set{0}$.
\end{theorem}
\begin{proof}
  Without loss of generality, suppose $\R/\Z\times B(a)\subset \Omega\subset \R/\Z\times B(R)$ for $1<a\le R$. Embed $\R/\Z\times B(R)$ into the ideal boundary of the aforementioned Liouville manifold $W$, and use this domain to compute the $Q$ groups.

  Because the composition of:
  \begin{equation*}
    Q(\R/\Z\times B(a);\kappa)\to Q(\Omega;\kappa)\to Q(\R/\Z\times B(R);\kappa)
  \end{equation*}
  is an isomorphism, $Q(\Omega;\kappa)$ must be non-zero. Since $\kappa$ is a non-trivial class (contains no contractible orbits), it follows from Theorem \ref{theorem:main} that $X$ has a positive orbit in the class of $\kappa$, as desired.
\end{proof}

We suspect that this existence theorem can be deduced by combining our Lemma \ref{lemma:mu-approach-barcode} with the techniques of \cite{ekp} (which are specifically tailored to prequantizations), or the generating function approach of \cite{sandon-ann-inst-four-2011,fraser-sandon-zhang} which are specifically tailored to the prequantization of $\C^{n}$. It is a natural question as to whether the hypotheses on the non-resonance of $\bd \Omega$ or the requirement that $\Omega$ contains $\R/\Z\times B(1)$ can be relaxed, although we leave this question for future research.

\subsection{A contact vector field with no orbits at all}
\label{sec:contact-vector-field-no-orbits-at-all}

The question which started this project was the following: {\itshape does every contact vector field on a compact contact manifold have a closed orbit?}

It is not so easy to come up with a counterexample; indeed, if $X$ is a contact vector field without orbits, then $X$ must be everywhere non-zero; note that a zero of $X$ counts as an orbit. In this case the dividing hypersurface of $\Sigma$ is cut transversally. Moreover, since $X$ has no orbits, and $X$ directs the characteristic foliation of $\Sigma$ (see \S\ref{sec:char-orbits-divid}), it follows that the characteristic foliation of $\Sigma$ has no closed orbits. Thus $X$ is non-resonant and our invariant can be applied. Up to changing $X$ to $-X$, the question reduces to $X$ having positive orbits, i.e., a generalization of Weinstein's conjecture.

However there are examples of contact vector fields without closed orbits. For instance, one can take any vector field on a closed manifold $N$ without closed orbits, and lift this vector field to a canonical vector field on $Y=ST^{*}N$. There are many examples of such $N$; see, e.g., the work of \cite{kuperberg-annals} for an example when $N=S^{3}$.

We note one curious feature of such examples. Any canonical vector field $X$ is generated by a contact Hamiltonian which is linear in each cotangent fiber. In particular, the anticontact involution $(p,q)\mapsto (-p,q)$ swaps the region where $X$ is positive and where $X$ is negative. Thus there is a certain symmetry between the positive and negative regions. Are there any contact vector fields without closed orbits which do not have such a symmetry?

\subsection{The non-squeezing theorem via selective Floer cohomology}
\label{sec:non-squeezing-via}

Continuing with the set-up of \S\ref{sec:case-prequantization-darboux-ball}, we now explain how the groups $Q(\Omega;\kappa)$ can be used to recover the famous contact non-squeezing result \cite[Theorem 1.2]{ekp}.

Recall that we had introduced a special Liouville manifold $W$, containing an embedded $\R/\Z\times B(R)$, which satisfied properties (1)-(3). In fact, the $W$ we consider satisfies an additional property:
\begin{enumerate}[resume]
\item\label{sec:item-4} the continuation morphism: $$Q(\R/\Z\times E(c,R,\dots,R);\kappa)\to Q(\R/\Z\times E(a,R,\dots,R);\kappa)$$ is zero if $c<1<a\le R$.
\end{enumerate}
Here we introduce the standard symplectic ellipsoid:
\begin{equation*}
  E(a_{1},\dots,a_{n})=\set{\sum a_{i}^{-1}\pi \abs{z_{i}}^{2}\le 1}.
\end{equation*}
As we will explain in \S\ref{sec:conj-isom}, the invariants $Q(\Omega;\kappa)$ also satisfy a sort of conjugation invariance:
\begin{lemma}
  If $\psi$ is a contactomorphism of the ideal boundary $Y$ which extends to a symplectomorphism of $W$, then there is an induced natural transformation:
  \begin{equation*}
    \psi_{*}:Q(\Omega;\kappa)\to Q(\psi(\Omega);\psi(\kappa));
  \end{equation*}
  the transformation is natural when $Q(-;\kappa)$ and $Q(\psi(-);\psi(\kappa))$ are considered as functors on the category of subdomains of the ideal boundary.
\end{lemma}

As a consequence of the lemma and item (4) we are able to recover the non-squeezing result of \cite{ekp}:
\begin{theorem}\label{theorem:our-non-squeezing}
  Let $W$ be a Liouville manifold for which there is a contact embedding $i:\R/\Z\times B(R)\to \bd W$ satisfying:
  \begin{enumerate}[label=(\roman*)]
  \item\label{item:ns-1} the central knot is not homotopic to any of its iterates in $W$,
  \item\label{item:ns-2} $c_{1}(TW)=0$,
  \item\label{item:ns-3} the central knot is primitive in $Y$, i.e., is not homotopic to any iterated loop,
  \end{enumerate}
  Then there is no contactomorphism $\psi$ of the ideal boundary $\bd W$ so that:
  \begin{enumerate}
  \item $\psi(\R/\Z\times B(a))\subset \R/\Z\times E(c,R,\dots)$,
  \item $\R/\Z\times E(a,R,\dots)\subset \psi(\R/\Z\times B(R))$,
  \end{enumerate}
  if $c<1<a\le R$.
\end{theorem}
\begin{proof}
  We first prove the case when $\psi$ is extendable to the filling. In this case we do not need to assume the central knot is primitive in $Y$.

  Let $\Omega_{1}=\R/\Z\times E(c,R,\dots)$, $\Omega_{2}=\R/\Z\times B(a)$, $\Omega_{3}=\R/\Z\times B(R)$, and let $\kappa$ denote the (saturated) free homotopy class of the central orbit.

  If there were such a $\psi$, we would have a commutative diagram:
  \begin{equation*}
    \begin{tikzcd}
      {Q(\Omega_{2};\kappa)}\arrow[d,"{\simeq}"]\arrow[rr,"{}"] &{}&{Q(\Omega_{3};\kappa)}\arrow[d,"{\simeq}"]\\
      {Q(\psi(\Omega_{2});\kappa)}\arrow[r,"{}"] &{Q(\Omega_{1};\kappa)}\arrow[r,"{0}"] &{Q(\psi(\Omega_{3});\kappa),}
    \end{tikzcd}
  \end{equation*}
  where the vertical maps are conjugation isomorphisms and the horizontal morphisms are continuation maps. The diagram is commutative because the conjugation isomorphisms are natural transformations. However, the top vertical map is an isomorphism by item \eqref{sec:item-3} while the lower map is clearly zero by item \eqref{sec:item-4}. This gives the desired contradiction.
\end{proof}

We note that \cite{ekp} does not have the requirement that $\psi$ extends to a symplectomorphism of $W$. We should remark that $\psi$ extending to a symplectomorphism is weaker than $\psi$ being contact isotopic to the identity. It is not known whether there are ``exotic'' contactomorphisms of $\R/\Z\times B(R)$ (indeed, such a problem is a sort of contact cousin of the well-known open question of whether there are exotic symplectomorphisms of $B(R)$; Gromov has shown that the latter question has a negative answer in dimension $4$).

It is potentially possible to remove the condition that $\psi$ is extensible to $W$ using certain hypertight contact manifolds; this sometimes allows one to work directly with symplectizations; see \cite{albers-fuchs-merry,meiwes-naef-hypertight}.

In this paper we argue in an ad hoc fashion to remove the assumption that $\psi$ extends to $W$; the argument is given in \S\ref{sec:proof-non-squeezing}.

\subsection{Comparison with the selective symplectic homology}
\label{sec:comp-with-select}

In \cite{uljarevic-ssh}, the \emph{selective symplectic homology}, denoted by $SH_{\Omega}(W)$, is defined via:\footnote{Contrary to \cite{uljarevic-ssh} we use cohomological conventions for defining the Floer differential and continuation morphisms, hence we write $SH_{\Omega}(W)$ rather than $SH^{\Omega}(W)$.}
\begin{equation}\label{eq:colim-lim-sh}
  SH_{\Omega}(W):=\colim_{h\in \mathscr{H}(\Omega)}\lim_{f\in \Pi(h)}HF(h+f),
\end{equation}
where $HF(h+f)$ is the Floer cohomology of the time $1$ map of the contact vector field associated to the contact Hamiltonian $h+f$, using a fixed contact form on the ideal boundary; here:
\begin{enumerate}
\item $\mathscr{H}(\Omega)$ is a family of contact Hamiltonian with compact support in $\Omega$; the colimit is defined by sending $h$ to $+\infty$.
\item $\Pi(h)$ is a family of perturbations; the limit over $f\in \Pi(h)$ is computed by sending $f$ to zero.
\end{enumerate}
From this definition, and abstract nonsense, there is a well-defined continuation morphism:
\begin{equation*}
  SH_{\mathrm{Int}(\Omega)}(W)\to Q(\Omega)\to SH(W).
\end{equation*}
However, because $Q(\Omega)$ uses contact Hamiltonians with a non-zero derivative on the boundary $\bd\Omega$ and $SH_{\mathrm{Int}(\Omega)}(W)$ uses contact Hamiltonians $h$ with compact support, an inverse morphism $Q(\Omega)\to SH_{\mathrm{Int}(\Omega)}(W)$ is not clearly well-defined via continuation.

However, if $U$ is any open set which contains $\Omega$ in its interior, then a continuation morphism:
\begin{equation*}
  Q(\Omega)\to SH_{U}(W)
\end{equation*}
is well-defined; one can therefore conclude a factorization:
\begin{equation*}
  SH_{\mathrm{Int}(\Omega)}\to Q(\Omega)\to \lim_{\Omega\subset U}SH_{U}(W)\to SH(W);
\end{equation*}
whether or not the first two morphisms are isomorphisms seems to be a slightly subtle question about the characteristic foliation on $\bd \Omega$ which we defer to future research.

\subsubsection{Comparison for convex domains}
\label{sec:comp-conv-doma}

One thing which is fairly obvious is that, if $\bd \Omega$ is \emph{convex}, then the morphism $SH_{\mathrm{Int}}(\Omega)\to Q(\Omega)$ is an isomorphism. This can be seen as follows: let $\Omega_{\sigma}$, $\sigma\in (-\epsilon,\epsilon)$, be the family obtained by flowing by a contact vector field which is outwarldy transverse to $\bd\Omega$. The continuation map $Q(\Omega_{-\epsilon})\to Q(\Omega)$ is an isomorphism, as follows from the argument in \S\ref{sec:criterion-for-isomorphism}; since there is a factorization:
\begin{equation*}
  Q(\Omega_{-\epsilon})\to SH_{\mathrm{Int}(\Omega)}(W)\to Q(\Omega),
\end{equation*}
one concludes that $SH_{\mathrm{Int}(\Omega)}(W)\to Q(\Omega)$ is surjective. Injectivity is proved in a similar fashion, using that continuation $SH_{\mathrm{Int}(\Omega)}(W)\to SH_{\mathrm{Int}(\Omega_{\epsilon})}(W)$ is an isomorphism; see \cite{uljarevic-ssh}.

\subsubsection{Existence of positive orbits}
\label{sec:exist-posit-orbits}

The comparison with $SH_{\mathrm{Int}(\Omega)}(W)$ and results in \cite{uljarevic-ssh} allow one to use Theorem \ref{theorem:existence} to conclude the existence of positive orbits in certain cases.

For instance, if $\Omega$ is the complement of tubular neighborhood of an \emph{immaterial transverse knot} then it is shown in \cite[Theorem 7.4]{uljarevic-ssh} that:
\begin{equation*}
  SH_{\mathrm{Int}(\Omega)}(W)\to SH(W)
\end{equation*}
is surjective. Because of the factorization $SH_{\mathrm{Int}(\Omega)}(W)\to Q(\Omega)\to SH(W)$, one concludes that $Q(\Omega)\to SH(W)$ is also surjective.

Consequently, if $H^{*}(W)\to SH(W)$ is \emph{not} surjective, then $H^{*}(W)\to Q(\Omega)$ cannot be an isomorphism and hence Theorem \ref{theorem:existence} implies any contact vector field adapted to $\Omega$ has positive orbits.

\subsection{Acknowledgements}
\label{sec:acknowledgements}

The authors benefitted from enlightening discussions with Y.~Eliashberg, E.~Kilgore, E.~Shelukhin, and J.~Zhang. The first named author was support by funding from the CIRGET research group and the Fondation Jacques Courtois. The second named author was supported by the Science Fund of the Republic of Serbia, grant no.~7749891, Graphical Languages - GWORDS.

\section{Non-resonant contact vector fields}
\label{sec:non-resonant-contact-vector-fields}

In this section we define \emph{non-resonance}. As explained below, non-resonance of $X$ is essentially a property about the dividing hypersurface $\Sigma$ of $X$, assuming $\Sigma$ is cut transversally by $\alpha(X)$. We show in \S\ref{sec:convex-hypersurfaces} that $X$ is non-resonant provided $\Sigma$ is a convex hypersurface.

\subsection{Free homotopy classes of orbits}
\label{sec:free-homot-class}

It is important in our applications to refine non-resonance by a free homotopy class of orbits. We call any collection $\kappa$ of connected components of the free loop space of $Y$ a \emph{free homotopy class}. In particular, we do not require $\kappa$ to be a single connected component.

\subsubsection{Saturated classes}
\label{sec:saturated-classes}

If $Y$ is the ideal boundary of a convex-at-infinity symplectic manifold $W$, then we say that $\kappa$ is \emph{saturated} provided that:
\begin{equation}\label{eq:saturation}
  \kappa=i^{-1}(i(\kappa)),
\end{equation}
where $i:\pi_{0}(\Lambda Y)\to \pi_{0}(\Lambda W)$ is the map induced by the canonical-up-to-homotopy inclusion of $Y$ into $W$ as a starshaped hypersurface. In other words, we require that $\gamma_{1},\gamma_{2}$ both lie in $\kappa$ if and only if $i(\gamma_{1})$ and $i(\gamma_{2})$ both lie in $i(\kappa)$.

If $K$ is some specific loop, then $i^{-1}(i(\set{K}))$ is called the \emph{saturation} of $K$, and it consists of all other loops which are homotopic to $K$ within $W$. We are often interested in the case when: $$K=\R/\Z\times\set{0}\subset \R/\Z\times B(R)\subset Y,$$ and it is important in our applications to assume that its saturation $\kappa$ contains no other iterate of $K$. If this happens then $\kappa$ does not contain any constant loops, since the $0$th iterate of $K$ is constant.

\subsection{Non-resonance}
\label{sec:non-resonance}

Roughly speaking, non-resonance is a dynamical property needed to isolate the behaviour of the orbits of a contact vector field $X$ in the positive region.

\subsubsection{Definition of non-resonance}
\label{sec:definition}

Let us say that a pair $(X,\alpha)$ of a contact vector field and contact form is \emph{non-resonant relative a free homotopy class $\kappa$} provided that $h=\alpha(X)$ cuts out the dividing hypersurface tranversally and $R^{\alpha}$ is transverse to the dividing set $\Sigma$ along every closed orbit of $X$ contained in the dividing set in the free homotopy class of $\kappa$.

We show in \S\ref{sec:char-orbits-divid} that $(X,\alpha)$ being non-resonant relative $\kappa$ depends only on the pair $(\Sigma,\alpha)$.

The definition of non-resonance is inspired by \cite[Definition 5.1]{ekp}, although our notion is weaker than theirs as we do allow some closed characteristic orbits.

\subsubsection{Characteristic orbits on the dividing hypersurface}
\label{sec:char-orbits-divid}

If $X_{1}$ and $X_{2}$ have the same dividing hypersurface $\Sigma$ (which is assumed to be cut transversally) then $h_{1}=\alpha(X_{1})$ and $h_{2}=\alpha(X_{2})$ are proportional, i.e., $h_{1}=f h_{2}$ for some non-vanishing function $f$. Notice that:
\begin{equation*}
  \alpha(X_{1}-fX_{2})=0.
\end{equation*}
The contact condition implies that, at points in $\Sigma$, we have:
\begin{equation*}
  \d\alpha(X_{1},-)=\d h_{1}(R)\alpha-\d h_{1}=f(\d h_{2}(R)\alpha-\d h_{2})=f\d\alpha(X_{2},-).
\end{equation*}
It follows that $\d\alpha(X_{1}-fX_{2},-)=0$ vanishes along $\Sigma$, and hence $X_{1}-fX_{2}$ vanishes identically on $\Sigma$.

In particular, the (singular) line field spanned by $X$ on $\Sigma$ depends only on the hypersurface $\Sigma$. This line field is known as the \emph{characteristic foliation} of $\Sigma$, and it can be characterized as the kernel of $\d\alpha$ restricted to $\xi\cap T\Sigma$.

\subsubsection{Independence of the choice of contact form}
\label{sec:indep-choice-cont}

The choice of contact form is rather flexible in the definition of non-resonance, as the following lemma makes precise.
\begin{lemma}
  The set of contact forms $\alpha$ (with the chosen coorientation) for which $(X,\alpha)$ is non-resonant is contractible; indeed, if $(X,\alpha_{1})$ and $(X,\alpha_{2})$ are non-resonant, then the contact form $\alpha_{3}$ whose Reeb vector field is a convex combination: $$\theta R^{\alpha_{1}}+(1-\theta)R^{\alpha_{2}}$$ satisfies $(X,\alpha_{3})$ being non-resonant. The statement also holds if one restricts to a free homotopy class $\kappa$.
\end{lemma}
\begin{proof}
  The fact that $X$ is a contact vector field implies that:
  \begin{equation*}
    \varphi_{s}^{*}\alpha=e^{g_{s}}\alpha,
  \end{equation*}
  where $\varphi_{s}$ is the flow by $X$ and:
  \begin{equation*}
    \bd_{s}g_{s}\circ \varphi_{s}^{-1}=\d h(R^{\alpha}).
  \end{equation*}
  In particular, for every $x$ we have:
  \begin{equation*}
    g_{s}(x)=\int_{0}^{s}\d h(R^{\alpha})(\varphi_{\tau}(x))\d \tau.
  \end{equation*}
  If $R^{\alpha}$ is transverse to the dividing set along the trajectory $\varphi_{s}(x)$, then $\d h(R^{\alpha})$ is either positive or negative along the orbit, and hence $g_{s}(x)>0$ or $g_{s}(x)<0$. Let us call closed orbits with $g_{s}(x)>0$ for $s>0$ \emph{p-type} and orbits with $g_{s}(x)<0$ for $s>0$ \emph{n-type}. If $(X,\alpha)$ is non-resonant then every orbit in the dividing hypersurface is either p-type or n-type.

  Notice that if $\varphi_{s}(x)$ is a closed orbit with period $s_{0}>0$, then:
  \begin{equation}\label{eq:contact-comparison}
    (\varphi_{s_{0}}^{*}\alpha)_{x}=e^{g_{s_{0}}(x)}\alpha_{x}.
  \end{equation}
  If $\varphi_{s}(x)$ is a p-type orbit, then $(\varphi_{s_{0}}^{*}\alpha)_{x}>\alpha_{x}$, and the reverse inequality holds for n-type orbits. However, \eqref{eq:contact-comparison} is independent of the choice of contact form since $\varphi_{s_{0}}(x)=x$; here we restrict to contact forms defining the chosen coorientation.

  Therefore, if $(X,\alpha_{1})$ and $(X,\alpha_{2})$ are both non-resonant, then $\d h(R^{\alpha_{1}})$ and $\d h(R^{\alpha_{2}})$ must have the same sign along the orbit $\varphi_{s}(x)$ (depending on whether $\varphi_{s}(x)$ is p-type or n-type). In particular, $\d h(\theta R^{\alpha_{1}}+(1-\theta)R^{\alpha_{2}})$ is non-vanishing, for any $\theta\in [0,1]$. This completes the proof.
\end{proof}

If $\Sigma$ is the cooriented dividing hypersurface of $X$ and $(X,\alpha)$ is non-resonant for some $\alpha$ relative $\kappa$, then we will simply say $\Sigma$ is \emph{non-resonant} relative $\kappa$.

\subsubsection{Example}
\label{sec:example}

The unit sphere in $\R^{2n+1}$ is non-resonant for the standard contact form $\alpha_{0}$, since the only closed characteristic orbits are constant orbits at the north or south pole, and $R^{\alpha_{0}}$ is transverse to the sphere at the north and south pole.

\subsubsection{Convex hypersurfaces}
\label{sec:convex-hypersurfaces}

The example in \S\ref{sec:example} can be generalized to the class of convex hypersurfaces. Recall that a \emph{convex hypersurface} is one which is transverse to some contact vector field; see, e.g., \cite{eliashberg-gromov-AMS-1991,giroux-CMH-1991,honda-huang-arXiv-2019,eliashberg-pancholi-arXiv-2022,salamon-lecture-notes-convex-hypersurfaces}.
\begin{lemma}
Every convex hypersurface $\Sigma$ is non-resonant.
\end{lemma}

\begin{proof}
  Let $Z$ be a contact vector field transverse to $\Sigma$. The proof of the lemma is based on two observations:
  \begin{enumerate}
  \item $\Sigma$ can be divided into two halves $\Sigma_{+}\cup \Sigma_{-}$ in such a way that $Z$ is positively transverse to $\xi$ on the interior of $\Sigma_{+}$ and negatively transverse to $\xi$ on the interior of $\Sigma_{-}$.
  \item There is a compact subset $K_{\pm}$ of the interior of $\Sigma_{\pm}$ so that every closed characteristic orbit lies in $K_{-}\cup K_{+}$. In other words, no closed characteristic orbit travels between the two halves.
  \end{enumerate}
  Because Reeb vector fields are closed under convex combinations, we can find a Reeb vector field $R^{\alpha}$ so that $R^{\alpha}=Z$ holds on $K_{+}$ and $R^{\alpha}=-Z$ holds on $K_{-}$. It then follows from (2) that $(X,\alpha)$ is non-resonant for any $X$ whose dividing surface is $\Sigma$, as desired.

  It remains to establish (1) and (2). First pick an auxiliary contact form $\alpha$. The decomposition in (1) follows by setting $k=\alpha(Z)$ and $\Sigma_{\pm}=\set{\pm k\ge 0}\cap \Sigma$. Denote: $$\Gamma=\set{k=0}\cap \Sigma,$$ so that $\Gamma$ is the common boundary of $\Sigma_{\pm}$.

  Now write $\Sigma$ as a transverse zero level set $\set{h=0}$, and let $X$ be the contact vector field so $\alpha(X)=h$.

  To establish (2), consider the contact vector field equations for $Z$ and $X$:
  \begin{equation*}
    \begin{aligned}
      &\d k+\d\alpha(Z,-)=\d k(R)\alpha,\\
      &\d h+\d\alpha(X,-)=\d h(R)\alpha.
    \end{aligned}
  \end{equation*}
  Insert $X$ into the first equation and $Z$ into the second equation and evaluate at points in $\Gamma$ to obtain:
  \begin{equation*}
    \d k(X)=-\d h(Z),
  \end{equation*}
  where we use that $\alpha(X)=\alpha(Z)=0$ holds along $\Gamma$. Since $Z$ is transverse to $\Sigma$, by assumption, $\d h(Z)$ is non-vanishing along $\Sigma$, and hence $\d k(X)$ is non-vanishing along $\Gamma$. Without loss, suppose that $\d k(X)>0$ holds along $\Gamma$; this can be achieved by reversing $Z$ on any components of $\Sigma$ where the sign is opposite. Then every orbit of $X$ passes from $\Sigma_{-}$ into $\Sigma_{+}$, and never travels back. Moreover, no closed orbit of $X$ enters a neighborhood of $\Gamma$; the complement of this neighborhood produces the desired compact sets $K_{\pm}$.

  As explained above, we modify $R^{\alpha}$ so that $R^{\alpha}$ agrees with $\pm Z$ in a neighborhood of $K_{\pm}$. This completes the proof.
\end{proof}

\subsubsection{Prequantizations of ellipsoids and non-resonance}
\label{sec:preq-ellips}

Introduce the standard symplectic ellipsoid in $\C^{n}$ associated to the vector $(a_{1},\dots,a_{n})$:
\begin{equation*}
  E(a)=\textstyle\set{\sum \pi a_{i}^{-1}\abs{z_{i}}^{2}\le 1};
\end{equation*}
we assume throughout that $0<a_{1}\le \dots\le a_{n}<\infty$. Let us define the \emph{prequantization} of this ellipsoid to be:
\begin{equation*}
  \Omega(a)=\R/\Z\times E(a);
\end{equation*}
considered as a compact domain in $\R/\Z\times \C^{n}$ with the contact form \eqref{eq:prequantization-form}. A straightforward computation shows that:
\begin{equation*}
  X_{a}=\pd{}{\theta}-\sum a_{i}^{-1}V_{i}
\end{equation*}
is a contact vector field when $V_{i}$ is the horizontal lift of the Hamiltonian vector field for $\pi \abs{z_{i}}^{2}$; the nice thing about the $V_{1},\dots,V_{n}$ vector fields is that they pairwise commute and each one defines an $\R/\Z$-action. Moreover:
\begin{equation*}
  h_{a}=\alpha(X_{a})=1-\sum \pi a_{i}^{-1}\abs{z_{i}}^{2},
\end{equation*}
satisfies that $\set{h_{a}\ge 0}=\Omega(a)$; i.e., $X_{a}$ is adapted to $\Omega(a)$.

Let $\kappa$ be the free homotopy class containing $\R/\Z\times\set{0}\subset \R/\Z\times \C^{n}$. Then:
\begin{lemma}
  If $a_{1}>1$, then $(X_{a},\alpha)$ is non-resonant relative the free homotopy class $\kappa$.
\end{lemma}
\begin{proof}
  Indeed, we will prove that $X_{a}$ has \emph{no} orbits in its dividing hypersurface which lie in the free homotopy class of $\kappa$. This is immediate; since the cross-sectional area $a_{i}$ is bigger than one, the vector field $a_{i}^{-1}Y_{i}$ rotates the $z_{i}$ coordinate by a total angle less than $2\pi$ (in time $1$).

  In particular, $X_{a}$ has no closed orbits of period less than $1$, except the central knot $z_{1}=\dots=z_{n}=0$ which never lies in the dividing hypersurface. Any orbit of period more than $1$ will clearly not lie in the free homotopy class $\kappa$; thus the proof is complete.
\end{proof}

Note that, if $\R/\Z\times B(R)\subset \bd W$, and if the saturation $\kappa$ of $K=\R/\Z\times \set{0}$ contains none of the other iterates of $K$, then $(X_{a},\alpha)$ is also non-resonant relative $\kappa$. This is because the dividing hypersurface of $X_{a}$ remains entirely in $\R/\Z\times B(R)$, and so the only orbits which can appear are homotopic to iterates of $K$.

\section{The selective Floer cohomology of a contact vector field}
\label{sec:defin-select-floer}

In this section we explain how to construct the selective Floer cohomology $Q(X;s;\kappa)$ associated to a contact vector field $X$ using the framework of \cite{ulja-zhang,djordjevic_uljarevic_zhang,cant-sh-barcode}. These selective Floer cohomology groups naturally form a persistence module whose colimit is a variant of the selective symplectic homology introduced in \cite{uljarevic-ssh}. The construction of these groups is carried out in \S\ref{sec:select-sympl-cohom}.

\subsection{Construction of the invariant}
\label{sec:constr-invar}

As explained in the introduction, we will construct cohomology groups $Q(X;s;\kappa)$, for $s\in [0,\infty)$ and a saturated free homotopy class $\kappa$ (see \S\ref{sec:free-homot-class}).

\subsubsection{Choice of cut-off function}
\label{sec:choice-cut-off}

Let $\mu:\R\to \R$ be a smooth convex function so that $\mu(x)=x$ for $x\ge 1$ and $\mu(x)=1/2$ for $x\le 0$, and let $\mu_{\delta}(x)=\delta\mu(x/\delta)$; see Figure \ref{fig:mu}. We think of $\delta\in (0,1)$ as a small parameter. Note that:
\begin{equation*}
  \ud{}{x}(\mu(x)-\mu_{\delta}(x))=\mu'(x)-\mu'(x/\delta)\le 0,
\end{equation*}
so $\mu(x)\ge \mu_{\delta}(x)$ holds for all $x\in \R$ (note the equality holds for $x\ge 1$). Consequently, $\mu_{\delta_{0}}(x)\ge \mu_{\delta_{1}}(x)$ holds for all $\delta_{0}\ge \delta_{1}$.

\subsubsection{Selecting the positive part of an isotopy}

Consider the contact vector field $X^{\alpha}_{\delta}$ generated by $\mu_{\delta}(h)$ and the contact form $\alpha$. A straightforward computation shows that:
\begin{equation}\label{eq:formula-contact-vf}
  X^{\alpha}_{\delta}=(\mu_{\delta}(h)-h\mu_{\delta}'(h))R^{\alpha}+\mu_{\delta}'(h)X,
\end{equation}
where $R^{\alpha}$ is the Reeb flow for $\alpha$. By construction $X^{\alpha}_{\delta}=X$ holds in the region where $\set{h\ge \delta}$.

\subsubsection{A genericity statement}
\label{sec:genericity}

In order to define the Floer cohomology of a contact isotopy $\psi_{t}$ using the framework of \cite{djordjevic_uljarevic_zhang,cant-sh-barcode,cant-hedicke-kilgore}, it is important that $\psi_{1}$ has no discriminant points. If $\psi_{t}$ is the autonomous flow by a positive contact vector field, this condition is that the contact vector field has no $1$-periodic orbits. The following claim will be sufficient for our constructions.

\begin{claim}
  For a generic set of times $s$, the positive contact vector field $X^{\alpha}_{\delta}$ has no $s$-periodic orbits.
\end{claim}
\begin{proof}
  It is clear that $X^{\alpha}_{\delta}$ is simply a Reeb flow for some contact form (indeed, every positive contact vector field is). Thus the result follows from the well-known fact that the spectrum of periods of a Reeb vector field is nowhere dense. This is due to Sard's theorem, since the periods of $R^{\alpha}$ are the critical values of the contact-action $\gamma\mapsto \int \gamma^{*}\alpha$; see, e.g., \cite{hofer-zehnder-90,schwarz-pjm-2000,ginzburg-2005-weinstein}.
\end{proof}

\subsubsection{Floer cohomology of certain contact isotopies}
\label{sec:floer-cohom-cont}

Recall from \cite{merry-uljarevic,djordjevic_uljarevic_zhang,cant-sh-barcode,cant-hedicke-kilgore} that to any contact isotopy $\set{\psi_{t}:t\in [0,1]}$, so that $\psi_{0}=\id$ and $\psi_{1}$ has no discriminant points, there is a Floer cohomology group $HF(\psi_{t})$ defined as the Floer cohomology of any contact-at-infinity Hamiltonian isotopy of $W$ whose ideal restriction is $\psi_{t}$. When $\kappa$ is a saturated class, one defines $HF(\psi_{t};\kappa)$ as the piece of Floer cohomology generated by orbits in the class of $i(\kappa)$ where $i:Y\to W$ is the inclusion of the ideal boundary.

In this paper we are primarily interested in the autonomous isotopies $\psi_{s}$ generated by the contact vector fields $X_{\delta}^{\alpha}$. In this case we write:
\begin{equation*}
  HF(X^{\alpha}_{\delta};s;\kappa)
\end{equation*}
to be the Floer cohomology of the contact isotopy $\psi_{st}$ in the class $\kappa$.

As explained in \S\ref{sec:genericity}, this Floer cohomology is well-defined for $s$ in the complement of a nowhere dense set. We can therefore extend the Floer cohomology to all $s\in \R$ by a limit:
\begin{equation}\label{eq:limit-completion}
  HF(X^{\alpha}_{\delta};s;\kappa)=\lim_{\epsilon\to 0+}HF(X^{\alpha}_{\delta};e^{\epsilon}s;\kappa),
\end{equation}
where the morphisms in the limit are given as continuation maps. We defer the definition of continuation maps until \S\ref{sec:continuation-maps}, but let us just say a few things for now:
\begin{enumerate}
\item Continuation maps are defined $HF(X_{0};s;\kappa)\to HF(X_{1};s;\kappa)$ using continuation data $X_{\sigma}$, namely, a one-parameter family of contact vector fields $X_{\sigma}$ interpolating from $X_{0}$ to $X_{1}$ which is \emph{non-decreasing}.
\item The non-decreasing condition is a convex condition to place on $X_{\sigma}$; consequently, the morphism is canonical and depends only on the existence of a non-decreasing interpolation. If $h_{0},h_{1}$ are contact Hamiltonians for $X_{0},X_{1}$, then a necessary and sufficient condition for there to exist a non-decreasing interpolation is $h_{0}\le h_{1}$.
\item $HF(e^{\epsilon}X;s;\kappa)=HF(X;e^{\epsilon}s;\kappa)$, so the continuation maps in (1) can be used to define the morphisms in the limit \eqref{eq:limit-completion}.
\end{enumerate}

\subsubsection{Definition of the selective Floer cohomology}

Define:
\begin{equation*}
  Q(X; s,\kappa):=\lim_{\alpha}\lim_{\delta\to 0+}HF(X^{\alpha}_{\delta};s;\kappa).
\end{equation*}
The inverse limits are computed with respect to continuation maps (see \S\ref{sec:continuation-maps}). Indeed, we recall that:
\begin{equation*}
  \mu_{\delta_{1}}(h)\le \mu_{\delta_{0}}(h)
\end{equation*}
if $\delta_{1}\le \delta_{0}$, and hence there is a continuation data from $X^{\alpha}_{\delta_{1}}$ to $X^{\alpha}_{\delta_{0}}$. The limit over contact forms is formal and is explained next.

\subsubsection{Dependence on the contact form}
\label{sec:depend-cont-form}

Suppose that $\lambda=e^{r}\alpha$ is a different contact form, and let $k=\lambda(X)=e^{r}h$. For $e^{r}\le \epsilon/\delta$ we have:
\begin{equation*}
  e^{r}\delta\mu(h/\delta)\le \epsilon \mu(e^{r}h/\epsilon),
\end{equation*}
so the linear interpolation from $X^{\alpha}_{\delta}$ to $X^{\lambda}_{\epsilon}$ is a non-negative path, and hence has a well-defined continuation map. It is with respect to these continuation maps and abstract nonsense that we obtain a canonical isomorphism:
\begin{equation*}
  \lim_{\delta\to 0}HF(X^{\alpha}_{\delta};s;\kappa)\to \lim_{\epsilon\to 0}HF(X^{\lambda}_{\epsilon};s;\kappa),
\end{equation*}
and it is these isomorphisms that are used in the limit over contact forms. The upshot of this discussion is that the vector space $Q(X;s;\kappa)$ is independent of the choice of $\alpha$, but it can be computed using any fixed contact form.

\subsubsection{The role of non-resonance}
\label{sec:role-non-resonance}

The following geometric lemma is where we use the non-resonance assumption.

\begin{lemma}\label{lemma:mu-approach-barcode}
  If $X_{\sigma}$ is a family of non-resonant contact vector fields relative a free homotopy class $\kappa$ then, for any $s_{0}>0$, there exists $\delta_{0}>0$ so that the following holds: every orbit of $X_{\sigma,\delta}^{\alpha}$ in the class of $\kappa$ with $\delta\le \delta_{0}$ and period at most $s_{0}$ is also an orbit of $X_{\sigma}$ and lies entirely in the region where $h_{\sigma}\ge \delta$.
\end{lemma}

\paragraph{}
\label{sec:criterion-for-isomorphism}

Before we give the proof, we explain how the lemma gives a criterion for establishing certain continuation maps are isomorphisms.

If $X_{\sigma}$ is a family of contact vector fields which satisfies:
\begin{equation*}
  X_{\sigma}\ge 0\implies \bd X_{\sigma}\ge 0
\end{equation*}
then it is easy to see $X_{\sigma,\delta}^{\alpha}$ is non-decreasing; thus there is an induced continuation map:
\begin{equation}\label{eq:partial-CM}
  HF(X^{\alpha}_{0,\delta};s;\kappa)\to HF(X^{\alpha}_{1,\delta};s;\kappa)
\end{equation}
which commutes with the other continuation morphisms, and consequently induces a morphism:
\begin{equation}\label{eq:general-CM}
  Q(X_{0};s;\kappa)\to Q(X_{1};s;\kappa)
\end{equation}
If $X_{\sigma}$ has no positive orbits of period $s$ in the class $\kappa$, and is non-resonant for each $\sigma$, then Lemma \ref{lemma:mu-approach-barcode} implies \eqref{eq:general-CM} is an isomorphism. This follows from the fact that there exists a family of contact vector fields (namely $X_{\sigma,\delta}^{\alpha}$) interpolating between $X_{0,\delta}^{\alpha}$ and $X_{1,\delta}^{\alpha}$ which is non-decreasing and which never has any orbits of period $s$ in the class $\kappa$. That this implies the continuation map \eqref{eq:partial-CM} is an isomorphism is proved in \cite{ulja-zhang}, and is an analog of well-known facts about continuation maps in other contexts; see \S\ref{sec:defin-cont-map-Q} for further discussion.

\paragraph{}

An important variant of the idea in \S\ref{sec:criterion-for-isomorphism} is the following: if $X$ is non-resonant and has no positive orbits of period in $[s_{0},s_{1}]$ in the class $\kappa$, then neither does $X_{\delta}^{\alpha}$ for $\delta$ sufficiently small. In particular, the continuation map:
\begin{equation}\label{eq:structure-map}
  Q(X;s_{0};\kappa)\to Q(X;s_{1};\kappa)
\end{equation}
is an isomorphism, completing the proof of Theorem \ref{theorem:pmod-barcode}. Understanding when this continuation map is an isomorphism is important for us as it plays the role of the structure map in the persistence module $Q(X;-;\kappa)$.

\begin{cor}
  If $X$ is non-resonant relative the class $\kappa$, then the endpoints of bars in the barcode for $Q(X;-;\kappa)$ form a subset of the periods of positive orbits of $X$ in the class of $\kappa$.
\end{cor}

\paragraph{}

We prove Lemma \ref{lemma:mu-approach-barcode}.

\begin{proof}
  Without loss of generality, let us prove the statement for a single vector field $X$; the generalization to a family $X_{\sigma}$ follows the same exact argument.

  We suppose $X$ is non-resonant relative the contact form $\alpha$ and class $\kappa$. See \S\ref{sec:indep-choice-cont} for how to pick contact forms when generalizing to parametric families of vector fields.

  For the purposes of the proof, let us refer to a positive orbit of $X^{\alpha}_{\delta}$ in the class $\kappa$ and period at most $s_{0}$ which passes through the region where $h\le \delta$ an \emph{extra} orbit. We need to show there is some $\delta_{0}$ so that, for $\delta\le \delta_{0}$, $X^{\alpha}_{\delta}$ has no extra orbits.

  Let $V$ be the unique vector field tangent to the contact distribution so that:
  \begin{equation}\label{eq:contact-vf}
    X^{\alpha}_{\delta}=\mu_{\delta}(h)R^{\alpha}+\mu'_{\delta}(h)V.
  \end{equation}
  It follows from the construction that $V\in \xi$ satisfies:
  \begin{equation*}
    \d\alpha(V,-)+\d h-\d h(R^{\alpha})\alpha=0,
  \end{equation*}
  and, in particular, $V$ preserves the level sets of $h$. For notational concision in the rest of the proof, we will omit the contact form $\alpha$ from the notation and write $X^{\alpha}_{\delta}=X_{\delta}$.

  Since $\d h(X_{\delta})=\mu_{\delta}(h)\d h(R)$, any flow line $\gamma(s)$ of $X_{\delta}$ satisfies:
  \begin{equation*}
    \pd{}{s}\mu_{\delta}(h(\gamma(s))=\mu_{\delta}'(h)\d h(R)\mu_{\delta}(h(\gamma(s))).
  \end{equation*}

  In particular, any flow line with $\mu_{\delta}(h(\gamma(s)))\le \delta$ for some $s\in [0,s_{0}]$ must remain entirely in the region $\mu_{\delta}(h)\le e^{Cs_{0}}\delta$ where:
  \begin{equation*}
    C=\max\abs{\d h(R)}\ge \max\abs{\mu_{\delta}'(h)\d h(R)}.
  \end{equation*}
  From another point of view, if $\mu_{\delta}(h(\gamma(s)))>e^{Cs_{0}}\delta$ holds for some $s\in [0,s_{0}]$, then $\mu_{\delta}(h(\gamma(s)))\ge \delta$ holds for all $s\in [0,s_{0}]$, in which case $\gamma(s)$ is an trajectory of the original vector field $X$, since:
  \begin{equation*}
    \mu_{\delta}(h)\ge \delta\implies \mu_{\delta}(h)=h\text{ and }\mu_{\delta}'(h)=1.
  \end{equation*}
  Thus, every extra orbit of $X_{\delta}$ remains in the region where $\mu_{\delta}(h)\le e^{Cs_{0}}\delta$.

  Since $\d h(X_{\delta})=\mu_{\delta}(h)\d h(R)$, we can conclude that any extra orbit of $X_{\delta}$ remains in the region $h\ge -Ce^{Cs_{0}}\delta s_{0}$ or remains entirely in the region where $h\le 0$. This is because:
  \begin{equation*}
    \abs{\mu_{\delta}(h)\d h(R)}\le Ce^{Cs_{0}}\delta.
  \end{equation*}

  However, in the region where $h\le 0$, $X_{\delta}=(\delta/2)R$. Picking $\delta/2$ smaller than the minimal period of a Reeb orbit we can therefore conclude that for $\delta$ sufficiently small every extra orbit of $X_{\delta}$ remains in the region where:
  \begin{equation*}
    h\in [-Cs_{0}e^{Cs_{0}}\delta,e^{Cs_{0}}\delta]\implies \mu_{\delta}(h)\in [\delta/2,e^{Cs_{0}}\delta].
  \end{equation*}

  We conclude the proof with a compactness argument: suppose that $x_{n},\delta_{n}$ are sequences so that $\delta_{n}\to 0$, and $\gamma_{n}(s)$ is an extra orbit of $X_{\delta_{n}}$ with initial condition $\gamma_{n}(0)=x_{n}$; we will then derive a contradiction using the non-resonance condition. Let $s_{n}\in [0,s_{0}]$ be the period of $\gamma_{n}$, and introduce:
  \begin{equation*}
    \begin{aligned}
      a_{n}(t)&=s_{n}\mu_{\delta_{n}}(h(\gamma_{n}(s_{n}t))),\\
      b_{n}(t)&=s_{n}\mu_{\delta_{n}}'(h(\gamma_{n}(s_{n}t))).
    \end{aligned}
  \end{equation*}
  Consider the non-autonomous vector fields:
  \begin{equation*}
    F_{n,t}=a_{n}(t)R^{\alpha}+b_{n}(t)V\text{ and }G_{n,t}=b_{n}(t)V,
  \end{equation*}
  so that $\eta_{n}(t)=\gamma_{n}(s_{n}t)$ solves $\eta_{n}'(t)=F_{n,t}(\eta_{n}(t))$. By construction, $\eta_{n}$ is a $1$-periodic orbit of $F_{n,t}$.

  By passing to a subsequence, we may suppose that $x_{n}\in Y$ converges to $x_{\infty}\in \Sigma$ and $s_{n}$ converges to $s_{\infty}$. Let $g_{n}(t)$ be the flow line of $G_{n,t}$ starting at $x_{\infty}$. In suitable local coordinates we estimate:
  \begin{equation*}
    \pd{}{t}(\eta_{n}(t)-g_{n}(t))=a_{n}(t)R^{\alpha}(\eta_{n}(t))+b_{n}(t)[V(\eta_{n}(t))-V(g_{n}(t))].
  \end{equation*}
  Taking norms yields:
  \begin{equation*}
    \left|\pd{}{t}(\eta_{n}(t)-g_{n}(t))\right|\le C_{1}\delta+C_{2}\left|\eta_{n}(t)-g_{n}(t)\right|,
  \end{equation*}
  where we use $\abs{b_{n}(t)}\le s_{n}$ and $\abs{a_{n}(t)}\le s_{n}Ce^{s_{0}}\delta$; the constants $C_{1},C_{2}$ are allowed to depend on $s_{n}$ (which is bounded). A standard application of a Gronwall type inequality then implies that $\eta_{n}(t)$ and $g_{n}(t)$ remain close for all times $t\in [0,1]$; how close depends on how close $x_{n}$ is to $x_{\infty}$, the size of $s_{n}$, and how small $\delta$ is; see, e.g., \cite{teschl_ODE}.

  Since $g_{n}'(t)=b_{n}(t)V(g_{n}(t))$ and $b_{n}(t)\in [0,s_{n}]$, $g_{n}(t)$ is a reparametrized flow line of $V=X$ on $\Sigma$ of parameter length at most $s_{n}$. Moreover, $g_{n}(1)$ converges to $x_{\infty}$, and so $x_{\infty}$ lies on a closed orbit of $X$ of period at most $s_{\infty}$ on the dividing hypersurface. Here we use that a sequence of flow lines of $V$ with bounded parameter lengths converges in $C^{0}$ (after a subsequence) to a flow line.

  Moreover, a straightforward construction associates to $g_{n}(t)$ a free homotopy class for $n$ large enough; one simply takes a small geodesic ball around $x_{\infty}$ and connects $g_{n}(1)$ to $g_{n}(0)$. The resulting free homotopy class is preserved under reparametrization, and by comparison with $\eta_{n}(t)$ we conclude that this class lies in $\kappa$. It then follows that $x_{\infty}$ lies on a closed orbit on the dividing hypersurface in the class of $\kappa$.

  However, since $\eta_{n}(t)=\gamma_{n}(s_{n}t)$ is close to $g_{n}(t)$, the non-resonance implies that $\d h(R_{\alpha})>0$ or $\d h(R_{\alpha})<0$ holds along the orbit of $\gamma_{n}(s)$. Thus:
  \begin{equation*}
    \pd{}{s}h(\gamma_{n}(s))=\mu_{\delta}(h)\d h(R^{\alpha}),
  \end{equation*}
  and $\mu_{\delta}(sh)\ge \delta/2$, we conclude that $h(\gamma_{n}(1))\ne h(\gamma_{n}(0))$, i.e., $\gamma_{n}$ is not a closed orbit. This contradiction completes the proof.
\end{proof}

\subsection{Continuation maps}
\label{sec:continuation-maps}

The continuation maps we consider are fairly standard; given non-decreasing continuation data $X_{\sigma}$, see \S\ref{sec:continuation-data}, we consider an extension $X_{\sigma,t}$ to the filling as a family of time-dependent Hamiltonian vector fields, and then count the rigid Floer continuation cylinders:
\begin{equation*}
  \left\{
    \begin{aligned}
      &u:\R\times \R/\Z\to W,\\
      &\bd_{s}u+J(u)(\bd_{t}u-X_{\beta(-s),t}(u))=0;
    \end{aligned}
  \right.
\end{equation*}
here $\beta$ is a standard cut-off function so $\beta(-s)$ is non-increasing. Since $X_{\sigma}$ is non-decreasing, one therefore has a priori energy bounds and the relevant Floer theoretic arguments go through. The details of this assertion are covered elsewhere; see, e.g., \cite[\S4]{merry-uljarevic}, \cite[\S2.2.3]{uljarevic-ssh}, \cite[\S2.3.6]{brocic-cant-shelukhin}, \cite[\S2.3.3]{alizadeh-atallah-cant}, and \cite{floer89-comm-math-phys,salamon-zehnder,abouzaid_monograph,uljarevic-2017-JSG}.

\subsubsection{Continuation data}
\label{sec:continuation-data}

Let $X_{\sigma}$, $\sigma \in [0,1]$, be a smooth deformation of contact vector fields so that:
\begin{equation}
  \label{eq:nn-cont-data}
  \text{$\bd_{\sigma}X_{\sigma}$ is a non-negative contact vector field.}
\end{equation}
Note that:
\begin{equation*}
  h_{\sigma}=\alpha(X_{\sigma})\implies \pd{h_{\sigma}}{\sigma}=\alpha(\bd_{\sigma}X_{\sigma}),
\end{equation*}
so that \eqref{eq:nn-cont-data} is equivalent to requiring the contact Hamiltonians are non-decreasing.

There is a natural homotopy relation one can place on such deformations, asking they are homotopic relative their endpoints through deformations satisfying \eqref{eq:nn-cont-data}. Let us call such a homotopy class \emph{continuation data}. Note that the set of contact vector fields, with continuation data as morphisms, forms a category in a manner similar to the construction of the fundamental groupoid of a space. Moreover, condition \eqref{eq:nn-cont-data} ensures that, between any two objects $X_{0}$ and $X_{1}$ there is either a unique morphism or there is no morphism at all; the relevant criterion is whether or not the inequality $h_{0}\le h_{1}$ holds pointwise.

Associated to such continuation data, we define a morphism:
\begin{equation}\label{eq:morphism-pmod}
  \mathfrak{c}:HF(X_{0};s;\kappa)\to HF(X_{1};s;\kappa);
\end{equation}
these morphisms make $HF(X;s,\kappa)$ functorial when the set of contact vector fields is endowed with a category whose morphisms are continuation data (whose composition is given by the associative concatenation operation on continuation data). As a special case of this functoriality, \eqref{eq:morphism-pmod} induces a morphism of persistence modules $HF(X_{0};-;\kappa)\to HF(X_{1};-;\kappa)$.

\subsection{Definition of the cut-off continuation map}
\label{sec:defin-cont-map-Q}

We now explain how to use the general framework in \S\ref{sec:continuation-maps} to define continuation maps: $$Q(X_{0};s;\kappa)\to Q(X_{1};s;\kappa).$$

Let $X_{\sigma}$ be a one-parameter family of contact vector fields so that:
\begin{equation}\label{eq:semipositivity}
  X_{\sigma}>0\implies \bd_{\sigma}X_{\sigma}\ge 0.
\end{equation}
Recall $X^{\alpha}_{\sigma,\delta}$ is the contact vector field generated by $\mu_{\delta}(h_{\sigma})$ and $\alpha$. Since:
\begin{equation*}
  \bd_{\sigma}\mu_{\delta}(h_{\sigma})=\mu_{\delta}'(h_{\sigma})\bd_{\sigma}h_{\sigma},
\end{equation*}
it follows that $\bd_{\sigma}X^{\alpha}_{\sigma,\delta}\ge 0$ holds everywhere on $Y$. In this case there is a well-defined continuation map $HF(X^{\alpha}_{0,\delta};s;\kappa)\to HF(X^{\alpha}_{1,\delta};s;\kappa)$.

Since continuation maps commute with continuation maps, we can pass to the limit over $\delta$ and conclude there exists an induced morphism:
\begin{equation*}
  Q(X_{0};s;\kappa)\to Q(X_{1};s;\kappa).
\end{equation*}
That these morphisms turn $X\mapsto Q(X;-,\kappa)$ into a functor valued in the category of persistence modules is then a straightforward consequence of the general fact that continuation maps commute with continuation maps.

\subsection{Selective symplectic cohomology as a colimit}
\label{sec:select-sympl-cohom}

As explained in the introduction, for any compact domain $\Omega$ we may define:
\begin{equation*}
  Q(\Omega;\kappa)=\lim_{X}\colim_{s\to \infty}Q(X;s;\kappa).
\end{equation*}
The limit is over contact vector fields satisfying $\set{\alpha(X)\ge 0}=\Omega$ and so that $\bd \Omega$ is a regular level set of $\alpha(X)$; such vector fields are called \emph{adapted} to $\Omega$. The limit over $X$ is just a formality because the natural maps:
\begin{equation*}
  Q(\Omega;\kappa)\to \colim_{s\to \infty}Q(X;s;\kappa)
\end{equation*}
are isomorphisms for any fixed choice of $X$. Indeed, the morphisms used to compute the limit are canonical isomorphisms defined by interleaving cofinal sequences.

Arguments of a similar nature show that, for any inclusion of compact domains $\Omega_{1}\subset \Omega_{2}$ there is an associated continuation morphism:
\begin{equation*}
  Q(\Omega_{1};\kappa)\to Q(\Omega_{2};\kappa),
\end{equation*}
which is functorial with respect to iterated inclusions $\Omega_{1}\subset \Omega_{2}\subset \Omega_{3}$.

\subsection{Conjugation isomorphisms}
\label{sec:conj-isom}

We recall the definition of conjugation isomorphisms at the low level of Floer cohomology in $W$, and then explain how to extract natural isomorphisms $Q(\Omega;\kappa)\to Q(\psi(\Omega);\psi(\kappa))$.

\subsubsection{Definition at a low level}
\label{sec:definition-at-low}

Let $\psi$ be a contact-at-infinity symplectomorphism of $W$ and suppose that $\kappa$ is a saturated free homotopy class of orbits in $Y$. We do not assume that $\psi=\psi_{1}$ is the time-1 map of a contact-at-infinity isotopy.

This map induces a conjugation isomorphism: if $\varphi_{t}$ is any contact-at-infinity Hamiltonian system, then so is $\psi\varphi_{t}\psi^{-1},$ and their Floer cohomology groups are identified via $x\mapsto \psi(x)$; here $x$ is a fixed point of $\varphi_{1}$. See \cite[\S5]{uljarevic-ssh} for further discussion.

\subsubsection{Taking the limits and colimits}
\label{sec:taking-limits-colim}

The conjugation isomorphisms are natural (i.e., commutes with continuation morphisms), and hence induce limiting isomorphisms:
\begin{equation}\label{eq:conjugation-1}
  HF(X;s;\kappa)\to HF(\d\psi X\psi^{-1};s;\psi(\kappa)).
\end{equation}
Taking the colimit over $s$ yields the desired natural isomorphism:
\begin{equation}\label{eq:conjugation-2}
  Q(\Omega;\kappa)\to Q(\psi(\Omega);\psi(\kappa)).
\end{equation}
Note that in \eqref{eq:conjugation-1} and \eqref{eq:conjugation-2} only the ideal restriction of $\psi$ is involved in the objects. It seems to be an interesting question as to whether or not the morphism depends on the extension of $\psi$ to the filling $W$.

\section{Prequantizations of a symplectic ellipsoid}
\label{sec:prequant-sympl-ellip}

The goal in this section is to compute the $Q(\Omega;\kappa)$ groups for the prequantization domains introduced in \S\ref{sec:case-prequantization-darboux-ball}, \S\ref{sec:non-squeezing-via}, and \S\ref{sec:preq-ellips}.

\subsection{Geometric set-up}
\label{sec:geometric-set-up}

As described in \S\ref{sec:case-prequantization-darboux-ball} and \S\ref{sec:non-squeezing-via}, we are interested in certain continuation morphisms $Q(\Omega_{1};\kappa)\to Q(\Omega_{2};\kappa)$ when $\Omega_{i}$ are prequantizations of ellipsoids $\R/\Z\times E(a_{1},\dots,a_{n})$ embedded into the ideal boundary of a Liouville manifold $W$ as a tubular neighborhood of a transverse knot $K$.

As explained in the introduction, we assume that the saturation $\kappa$ of $K$ contains no other iterates of $K$. In particular, if a loop $\gamma$ is homotopic to $K$ inside of $W$ then $\gamma$ is non-contractible in $W$, so it is guaranteed that:
\begin{equation*}
  Q(X;0;\kappa)=0,
\end{equation*}
for any vector field adapted to $\Omega$.

\subsection{Examples of Liouville manifolds satisfying our requirements}
\label{sec:exampl-liouv-manif}

We now construct a Liouville manifold $W$ satisfying properties \ref{item:ns-1}, \ref{item:ns-2}, \ref{item:ns-3} from the statement of Theorem \ref{theorem:our-non-squeezing}.

We take $W$ to be a so-called divisor complement in a compact symplectic manifold $M$ which:
\begin{enumerate}
\item is symplectically aspherical,
\item has a vanishing first Chern class, i.e., there is a non-vanishing section of the complex determinant line $\det_{\C}(TM)$, and
\item has no torsion in its fundamental group;
\end{enumerate}
one can take, e.g., $M=T^{2n}$. Let us denote the divisor by $N\subset M$.

It is known that the ideal boundary $\bd W$ of the divisor complement is a prequantization of the divisor see, e.g., \cite{biran-lagrangian-barriers}. In particular, there is a circle bundle $\bd W\to N$ where $N$ is a compact symplectic manifold. Taking any Darboux ball $B(R)\subset N$ we obtain a fairly explicit embedding $\R/\Z\times B(R)\subset \bd W$. Moreover, by the divisor construction, the central knot $K$ actually bounds a disk in $M$ which intersects the divisor once transversally.

We establish property \ref{item:ns-1}. If $K^{i}$ and $K^{j}$ are homotopic in $W$, for $i\ne j \in \Z$, then there is a cylinder $C$ joining $K^{i}$ to $K^{j}$. Connect $C$ to disks in $M$ capping $K^{i}$ and $K^{j}$ which intersect the divisor $i$ times and $j$ times, respectively, to form a sphere $S$. By definition, the divisor is Poincaré dual to a non-zero multiple of the symplectic form in $M$, thus $S$ has symplectic area $k(i-j)$, where $k$ is a non-zero number. Since $M$ is assumed to be symplectically aspherical, we must have $i=j$, as desired.

From the construction it is immediate that $c_{1}(TW)=0$, i.e., \ref{item:ns-2} holds.

To establish \ref{item:ns-3}, we suppose, if $K\sim L^{k}$ for some other loop $L\subset \bd W$, then the inclusion of $L$ into $M$ is torsion (this is because the inclusion of $K$ into $M$ is contractible). By our assumption, this means that $L$ bounds a disk $D$ in $M$. This disk $D$ intersects the divisor $N$ some number of times, say $a\in \Z$. But since $L^{k}$ is homotopic to $K$ within $\bd W$, it follows that $K$ bounds a disk in $M$ which intersects $N$ $ka$ times. Because $M$ is aspherical it follows that $ka=1$, and hence $k=\pm 1$, as desired.

\subsection{A stability result}
\label{sec:stab-beyond-capac}

Recall the domains: $$\Omega(a_{1},\dots,a_{n})=\R/\Z\times E(a_{1},\dots,a_{n}),$$ with $a_{1}\le \dots \le a_{n}$. We will prove the following result which asserts that the $Q$ groups in class $\kappa$ stabilize:
\begin{lemma}
  Let $W,\kappa$ be as in \S\ref{sec:exampl-liouv-manif}. If $1<a_{1}\le b_{1}$ and $a_{j}\le b_{j}$ for all other $j$, then the continuation:
  \begin{equation*}
    Q(\Omega(a);\kappa)\to Q(\Omega(b);\kappa)
  \end{equation*}
  is an isomorphism.
\end{lemma}
\begin{proof}
  As explained in \S\ref{sec:preq-ellips}, the hypersurfaces $\bd \Omega((1-\sigma)a+\sigma b)$ are non-resonant relative $\kappa$ for each $\sigma\in [0,1]$. The vector fields $X_{(1-\sigma)a+\sigma b}$ from \S\ref{sec:preq-ellips} are adapted to $\Omega((1-\sigma)a+\sigma b)$ and never develop any orbits in the class of $\kappa$ with period $s>1$. Therefore Lemma \ref{lemma:mu-approach-barcode} implies the cut-off versions of $X_{(1-\sigma)a+\sigma b}$ never develop any orbits in the class of $\kappa$ with period $s>1$, and hence the continuation maps:
  \begin{equation*}
    Q(X_{a};s;\kappa)\to Q(X_{b};s;\kappa)
  \end{equation*}
  are isomorphisms as long as $s>1$; see \S\ref{sec:criterion-for-isomorphism} for related discussion.
\end{proof}

\subsection{A non-vanishing result}
\label{sec:non-vanishing-result}

Our next result shows that $Q(\Omega(a);\kappa)$ is non-zero, and for simplicity we work in the stable range $1<a_{1}$.

\begin{lemma}\label{lemma:non-vanishing-result}
  For $1<a_{1}$, there is an isomorphism $Q(\Omega(a);\kappa)\simeq \Z/2\oplus \Z/2.$
\end{lemma}

Before we present the proof, we introduce some preliminary notions.

\subsubsection{Local Floer cohomology}
\label{sec:local-floer-cohom}

Let $X_{\sigma}$ be a \emph{non-decreasing} family of contact vector fields, defined for $\sigma\in [0,1]$, and suppose that $X_{\sigma}$ develops a simple transversally non-degenerate $1$-periodic positive orbit $\Gamma$ when $\sigma=1/2$.

Moreover, suppose that $\Gamma$ is the only $1$-periodic orbit in the free homotopy class $\kappa$, and $\bd_{\sigma} X_{\sigma}$ is strictly positive along $\Gamma$.

Define the \emph{local Floer cohomology} of $\Gamma$ in the class $\kappa$ to be the cone of the continuation map:
\begin{equation*}
  CF(X_{0};\kappa)\to CF(X_{1};\kappa);
\end{equation*}
here the Floer complexes are defined by certain (time-dependent) extensions of the vector fields to the filling, and measure the fixed points of the time-1 maps; see \cite{cant-ulja-zhang} for further discussion.

\begin{lemma}
  If $\Gamma$ is simple and transversally non-degenerate, and $W$ is a Liouville manifold, then the local Floer cohomology of $\Gamma$ depends only on the restriction of $X_{\sigma}$ to a small open neighborhood of $\Gamma$.

  More precisely, if $W_{i}$, $X^{i}_{\sigma}$, $\Gamma_{i}$, $i=1,2$, are data as above, and there is a contactomorphism between tubular neighborhoods of $\Gamma_{i}\subset Y_{i}$ which identifies the contact vector fields $X^{i}_{\sigma}$ in those neighborhoods, then the local Floer cohomologies associated to the two paths are isomorphic as vector spaces over $\Z/2$.
\end{lemma}
\begin{proof}
  Such a result, relating the cone of a continuation morphism to a local Floer cohomology, is well-known. The notion of local Floer cohomology goes back to \cite{floer89-JDG,floer89-comm-math-phys,salamon-zehnder,cieliebak-floer-hofer-wysocki-1996,pozniak}, and is well studied when $\Gamma$ is an isolated set; see \cite{ginzburg-gurel-gt-09,ginzburg-gurel-JSG-2010,ginzburg-gurel-conley,shelukhin-zhao,shelukhin-hofer-zehnder,atallah-shelukhin-20}; some of these references discuss the local Floer cohomology for more general sets $\Gamma$. The work of \cite{cieliebak-floer-hofer-wysocki-1996,bourgeois-oancea-invent-09,bourgeois-oancea-duke-09} describes the local Floer cohomology for a transversally non-degenerate orbit $\Gamma$ using Morse-Bott methods; see also \cite{mclean-local-floer-GT-12,fehnder-local-symplectic-homology-2020,ginzburg-gurel-zeit-2020}. The proof of the stated Lemma appears in the upcoming joint work \cite{cant-ulja-zhang} in a more general context.
\end{proof}

\subsubsection{Proof of Lemma \ref{lemma:non-vanishing-result}}
\label{sec:proof-of-non-vanishing}

The argument uses local Floer cohomology.
\begin{proof}
  Fix $1<a_{1}\le \dots \le a_{n}\le R$, for some large $R>0$, and consider the family of Hamiltonian functions on $\C^{n+1}$:
  \begin{equation*}
    G_{\sigma}:=f(\sigma)\pi \abs{z_{0}}^{2}-(1+\epsilon)\sum_{j=1}^{n} \pi a_{j}^{-1}\abs{z_{j}}^{2},
  \end{equation*}
  where $f(\sigma)$ increases from $1-\epsilon$ to $1+\epsilon$ with a single transverse crossing at $f(1/2)=1$. Let $Y_{\sigma}$ denote the corresponding Hamiltonian vector field.

  Since $a_{j}$ are greater than $1$, the ideal restriction of $Y_{\sigma}$ has no $1$-periodic orbits except for when $\abs{z_{1}}=\dots=\abs{z_{n}}=0$, as long as $\epsilon$ is small enough. On that complex line (parametrized by $z_{0}$), $Y_{\sigma}$ acts as a rotation by total angle $2\pi f(\sigma)$, and therefore has a $1$-periodic orbit if and only if $f(\sigma)\in \Z$, which holds only when $f(1/2)=1$ by assumption. Therefore the hypotheses of the local Floer cohomology set-up apply.

  Our goal is to use $Y_{\sigma}$ to compute the local Floer cohomology of the orbit which develops at $\sigma=1/2$. Then we will implant this local model inside of the ideal boundary of the other Liouville manifold $W$ to compute $Q(\Omega(a);\kappa)$.

  For $\sigma\ne 1/2$, $Y_{\sigma}$ has a single non-degenerate orbit at $0\in \C^{n+1}$. Therefore the Floer cohomology is easily computed: $HF(Y_{\sigma})\simeq \Z/2$ for $\sigma=0,1$. Moreover, a computation of the Conley-Zehnder indices shows that the index of the orbit at $0$ is equal to:
  \begin{equation*}
    n+1+2\lfloor f(\sigma)\rfloor+2\sum_{j=1}^{n}\lfloor 1/a_{j}\rfloor=n+1+2\lfloor f(\sigma)\rfloor;
  \end{equation*}
  see, e.g., \cite[\S1.2.4]{cant-sh-barcode}. In particular, since the Conley-Zehnder index of the orbit at $0$ is different for $\sigma=0,1$, it follows that $HF(Y_{0})\to HF(Y_{1})$ is zero, and hence the local homology of the orbit which develops at $\sigma=1/2$ must be $\Z/2\oplus \Z/2$.

  Now consider the contact type embedding:
  \begin{equation*}
    \left\{
      \begin{aligned}
        &\R/\Z\times B(R)\to \C^{n+1},\\
        &\iota:(\theta,z)\mapsto (\frac{e^{2\pi i\theta }}{\pi^{1/2}},z);
      \end{aligned}
    \right.
  \end{equation*}
  this embedding is ``contact type'' because $\iota^{*}\Lambda=\d\theta+\lambda$ where $\Lambda$ is the radial Liouville form on $\C^{n+1}$; in particular, if we project $\iota$ to the ideal boundary of $\C^{n+1}$ it becomes a contact embedding of the prequantization into $S^{2n+1}$.

  With respect to this embedding, the contact Hamiltonian generating the ideal restriction of $Y_{\sigma}$ is: $$g_{\sigma}=G_{\sigma}\circ \iota=f(\sigma)-(1+\epsilon)\sum_{j=1}^{n}\pi a_{j}^{-1}\abs{z_{j}}^{2}.$$

  The single $1$-periodic orbit that $Y_{\sigma}$ develops is $\Gamma=\iota(\R/\Z\times\set{0})$; therefore we can use the above computation of the cone of $CF(Y_{0})\to CF(Y_{1})$ to conclude that the local homology of $\Gamma$ is $\Z/2\oplus \Z/2$.

  In other words, if there is some other contactomorphism $\R/\Z\times B(R)\to \bd W$, for another Liouville manifold, then any non-decreasing family of contact vector fields $X_{\sigma}$ on $\bd W$ such that:
  \begin{enumerate}
  \item $X_{\sigma}=Y_{\sigma}$ when pulled back to $\R/\Z\times B(r)$,
  \item the only $1$-periodic orbit $X_{\sigma}$ develops in the class of $\kappa$ is the central orbit $\R/\Z\times \set{0}$ at time $\sigma=1/2$ (which must exist since $X_{\sigma}=Y_{\sigma}$ holds on a neighborhood of this orbit),
  \end{enumerate}
  will necessarily induce a continuation map whose cone is $\Z/2\oplus \Z/2$ (in the class $\kappa$). We also note that $r$ can be taken as small as desired.

  Let $W$ be as in the statement, namely, suppose there is an embedding of $\R/\Z\times B(R)$ whose central orbit has a non-trivial saturation $\kappa$.

  For simplicity, let us suppose $a_{n}<R$. Pick a contact form $\alpha$ on $\bd W$ which agrees with $\alpha=\d\theta+\lambda$ in the neighborhood $\R/\Z\times W$, and let $X_{\sigma}=Y_{\sigma}$ in this neighborhood. Extend $X_{\sigma}$ to the complement as a negative contact vector field; this extension will not matter at all, since we are interested in the cut-off version $X_{\sigma,\delta}^{\alpha}$.

  Note that $\alpha(X_{\sigma})=0$ if and only if:
  \begin{equation*}
    \sum_{j=1}^{n}\frac{(1+\epsilon)\pi\abs{z_{j}}^{2}}{a_{j}f(\sigma)}=1,
  \end{equation*}
  and if $1+\epsilon<a_{1}f(\sigma)$, then we know from \S\ref{sec:preq-ellips} that this dividing hypersurface is non-resonant relative $\kappa$. Thus if $(1+\epsilon)<a_{1}(1-\epsilon)$, the dividing hypersurface will be non-resonant; this can certainly be achieved by taking $\epsilon$ small enough.

  Then we can apply the results in \S\ref{sec:role-non-resonance} to conclude that for $\delta$ sufficiently small the only $1$-periodic orbits $X_{\sigma,\delta}^{\alpha}$ develops in the class of $\kappa$ are the positive $1$-periodic orbits of $X_{\sigma}$ in the class of $\kappa$. By the locality of local Floer cohomology, it follows that the cone of a chain-level continuation morphism:
  \begin{equation*}
    CF(X_{0,\delta}^{\alpha};1;\kappa)\to CF(X_{1,\delta}^{\alpha};1;\kappa)
  \end{equation*}
  is $\Z/2\oplus \Z/2$. Since $X_{0}$ is non-resonant relative $\kappa$ and $\eta X_{0}$ never develops positive orbits in the class $\kappa$ for $\eta\in [0,1]$, we conclude that:
  \begin{equation*}
    HF(X_{0,\delta}^{\alpha};1;\kappa)=0,
  \end{equation*}
  and hence $HF(X_{1,\delta}^{\alpha};1;\kappa)\simeq \Z/2\oplus \Z/2$. However:
  \begin{equation*}
    X_{1,\delta}^{\alpha}=(1+\epsilon)X_{a,\delta}^{\alpha},
  \end{equation*}
  where $X_{a}$ is the standard contact-vector field adapted to $\Omega(a)$. Thus:
  \begin{equation*}
    Q(X_{a};1+\epsilon;\kappa)\simeq HF(X_{1,\delta}^{\alpha};1;\kappa)\simeq \Z/2\oplus \Z/2,
  \end{equation*}
  where we are assuming that $\delta$ is sufficiently small.

  Finally, since $X_{a}$ is non-resonant and never develops orbits in the class $\kappa$ with period bigger than $1$, we have:
  \begin{equation*}
    Q(\Omega(a);\kappa)\simeq Q(X_{1};1+\epsilon;\kappa)\simeq \Z/2\oplus \Z/2.
  \end{equation*}
  This completes the proof of the existence of the isomorphism.
\end{proof}

\subsubsection{Application to Theorem \ref{theorem:existence}}
\label{sec:appl-theor-existence}

At this point, we have proved everything necessary to conclude Theorem \ref{theorem:existence}, i.e., the existence of positive orbits for any contact vector field which is adapted to a domain $\Omega\subset \R/\Z\times \C^{n}$ with non-resonant boundary and which contains $\R/\Z\times B(1)$ in its interior.

\subsection{A vanishing result}
\label{sec:vanishing-result}

Using the set-up in \S\ref{sec:non-squeezing-via}, we have:
\begin{lemma}\label{lemma:vanishing}
  The continuation map:
  \begin{equation*}
    Q(\Omega(c,R,\dots);\kappa)\to Q(\Omega(a,R,\dots);\kappa)
  \end{equation*}
  is zero if $c<1<a\le R$.
\end{lemma}
This is used to show the non-squeezing result, Theorem \ref{theorem:our-non-squeezing}. Let us abbreviate $E=\Omega(c,R,\dots)$ and $F=\Omega(a,R,\dots)$. Without loss of generality, we will assume that $c\in (1/2,1)$ and $R$ is much larger than $1$.

\subsubsection{Local Floer homology for the ellipsoid $E$}
\label{sec:local-for-E}

In order to prove the vanishing result, we will again appeal to the local Floer homology. Arguing as in \S\ref{sec:proof-of-non-vanishing} one obtains:
\begin{equation*}
  Q(X_{E};1+\epsilon;\kappa)\simeq \text{local Floer homology of }X_{E,\sigma,\delta}^{\alpha}\text{ in class $\kappa$},
\end{equation*}
where $X_{E,\sigma}$ is the path of contact vector fields whose contact Hamiltonians are:
\begin{equation*}
  f(\sigma)-(1+\epsilon)\bigg(\frac{\pi \abs{z_{1}}^{2}}{c}+\sum_{j>1}\frac{\pi \abs{z_{j}}^{2}}{R}\bigg),
\end{equation*}
where $f(\sigma)$ increases from $1-\epsilon$ to $1+\epsilon$. Here we note that when $\sigma=1$ we have that $X_{E,\sigma}=(1+\epsilon)(X_{E})$ where $X_{E}$ is adapted to $E$. We briefly recall the construction.

First, one shows that $X_{E,\sigma}$ is non-resonant relative $\kappa$ provided the Hamiltonian system generated by:
\begin{equation*}
  H_{\sigma}=\sum\frac{\pi \abs{z_{j}}^{2}}{a_{j}(\sigma)}
\end{equation*}
has no closed orbits of period $1$ where:
\begin{equation*}
  a_{1}(\sigma)=\frac{cf(\sigma)}{1+\epsilon}\text{ and }a_{j}(\sigma)=\frac{Rf(\sigma)}{1+\epsilon};
\end{equation*}
this is equivalent to requiring that $a_{j}(\sigma)$ does not lie in $\set{1,1/2,1/3,\dots}$. Since $c\in (1/2,1)$, we can pick $\epsilon$ small enough so that:
\begin{enumerate}
\item $cf(\sigma)/(1+\epsilon)\in (1/2,1)$ and
\item $Rf(\sigma)/(1+\epsilon)>1$
\end{enumerate}
hold for all $\sigma$. This implies $X_{E,\sigma}$ is non-resonant, and the same argument given in \S\ref{sec:proof-of-non-vanishing} shows the local Floer homology of $X^{\alpha}_{E,\sigma,\delta}$ is supported on the single $1$-periodic orbit $\Gamma=\set{z=0}$ and is isomorphic to $Q(X_{E};1+\epsilon;\kappa)$.

The same argument given in the proof of \S\ref{lemma:non-vanishing-result} shows that there is a commutative square:
\begin{equation*}
  \begin{tikzcd}
    {\Z/2\oplus \Z/2\simeq Q(X_{E};1+\epsilon;\kappa)}\arrow[d,"{}"]\arrow[r,"{}"] &{Q(E;\kappa)}\arrow[d,"{}"]\\
    {\Z/2\oplus \Z/2\simeq Q(X_{F};1+\epsilon;\kappa)}\arrow[r,"{}"] &{Q(F;\kappa)},
  \end{tikzcd}
\end{equation*}
where the horizontal morphisms are isomorphisms. Even though both groups are isomorphic to $\Z/2\oplus \Z/2$, the vertical morphism is actually zero in this case. Roughly speaking, this is due to a shift in the Conley-Zehnder indices.

\subsubsection{Proof of lemma \ref{lemma:vanishing}}
\label{sec:proof-lemma-vanishing}

The rough outline of the argument is straightforward: one computes the induced map between cones working first in $\C^{n+1}$, and shows it vanishes for index reasons. One then uses the assumption that $c_{1}(TW)$ vanishes to deduce the map between cones in the other Liouville manifold $W$ also vanishes.

\begin{proof}
  In $\C^{n+1}$, we work with the Hamiltonians:
  \begin{equation*}
    G_{\eta,\sigma}=f(\sigma)\pi \abs{z_{0}}^{2}-(1+\epsilon)\bigg(\frac{\pi\abs{z_{1}}^{2}}{a_{1}(\eta)}+\sum_{j>1}\frac{\pi\abs{z_{j}}^{2}}{R}\bigg),
  \end{equation*}
  and let $Y_{\eta,\sigma}$ denote the corresponding Hamiltonian vector fields. Here:
  \begin{enumerate}
  \item $f$ increases from $1-\epsilon$ to $1+\epsilon$, and
  \item $\eta$ increases from $c$ to $a$.
  \end{enumerate}
  The restriction of $Y_{\eta,\sigma}$ to the contact-type hypersurface $\pi\abs{z_{0}}^{2}=1$, as in Lemma \ref{lemma:non-vanishing-result}, has the local model:
  \begin{equation}\label{eq:two-param-local-model}
    g_{\eta,\sigma}=f(\sigma)-(1+\epsilon)\bigg(\frac{\pi \abs{z_{1}}^{2}}{a_{1}(\eta)}+\sum_{j>1}\frac{\pi \abs{z_{j}}^{2}}{R}\bigg).
  \end{equation}
  Importantly, we observe that $Y_{0,\sigma}$ and $Y_{1,\sigma}$ only develop a $1$-periodic orbit when:
  \begin{equation*}
    \text{$f(\sigma)=1$ and $z_{1}=\dots=z_{n}=0$,}
  \end{equation*}
  which happens when $\sigma=1/2$.

  Since $\bd_{\eta}G_{\eta,\sigma}$ and $\bd_{\sigma}G_{\eta,\sigma}$ are non-negative, it follows that one has a square of continuation maps:
  \begin{equation*}
    \begin{tikzcd}
      {HF(Y_{0,0})}\arrow[d,"{}"]\arrow[r,"{}"] &{HF(Y_{0,1})}\arrow[d,"{}"]\\
      {HF(Y_{1,0})}\arrow[r,"{}"] &{HF(Y_{1,1})},
    \end{tikzcd}
  \end{equation*}

  The four groups $HF(Y_{0,0})$, $HF(Y_{1,0})$, $HF(Y_{0,1})$ and $HF(Y_{1,1})$ are all isomorphic to $\Z/2$. Moreover, each group is supported in a specific Conley-Zehnder index:
  \begin{equation*}
    CZ(Y_{\sigma,\eta})=n+1+2\lfloor f(\sigma)\rfloor+2\lfloor a_{1}(\eta)^{-1}(1+\epsilon)\rfloor;
  \end{equation*}
  as argued in the proof of Lemma \ref{lemma:non-vanishing-result}.

  We digress for a moment on the grading conventions in a cone. Our cones are computed as follows: find a chain level representative of a continuation morphism which is the inclusion of a subcomplex and then take the quotient complex. To compute the cone of $Y_{0,0}\to Y_{1,0}$, we therefore interpolate in such a way so that the generator of $Y_{0,0}$ survives the continuation morphism; let us call $Y_{\sigma,0}'$ the resulting deformation. Then $CF(Y_{1,0}')$ will necessarily have additional generators compared to $CF(Y_{1,0})$ because some are needed to cancel the generators from $CF(Y_{0,0})$. Since the cohomological differential decreases the Conley-Zehnder index, these new generators have index equal to $CZ(Y_{0,0})+1$. The cohomology of $Y_{1,0}$ also injects into the cone, and this contribution lives in the graded piece $CZ(Y_{0,0})+2$. In this way, we have shown that the local Floer cohomology of the path $Y_{\sigma,0}$ equals $\Z/2$ in each graded piece $CZ(Y_{0,0})+i$, $i=1,2$.

  The same argument proves that the local Floer cohomology of the path $Y_{\sigma,1}$ equals $\Z/2$ in each graded piece $CZ(Y_{0,1})+i$, $i=1,2$. However, we have computed that $CZ(Y_{0,0})-CZ(Y_{0,1})=2$ and so the cones are supported in different graded pieces.

  Because the relative Conley-Zehnder index in $\C^{n}$ is a local quantity, it follows that the cones of $Y_{0,\sigma,\delta}^{\alpha}$ and $Y_{1,\sigma,\delta}^{\alpha}$ on $S^{2n-1}$ are also supported in different graded pieces. This is because the paths $Y_{i,\sigma,\delta}^{\alpha}$ and $Y_{i,\sigma}$ compute the same local Floer cohomology, for $i=0,1$.

  Consider another Liouville manifold $W$ with $\R/\Z\times B(R)\subset W$, so that:
  \begin{enumerate}
  \item the class $\kappa$ of the central knot $K$ contains no other iterates of $K$,
  \item $c_{1}(TW)=0$,
  \end{enumerate}
  Let $X_{\eta,\sigma}$ be a vector field which agrees with local model \eqref{eq:two-param-local-model} in $\R/\Z\times B(R)$, and is non-positive outside, and let $X^{\alpha}_{\eta,\sigma,\delta}$ be the cut-off version, using a contact form $\alpha$ which agrees with the standard form on $\R/\Z\times B(R)$.

  For $\delta$ sufficiently small, the cones of $X^{\alpha}_{i,\sigma,\delta}$, $i=0,1$ are local Floer cohomologies of the orbits which develop along the central kont $z_{1}=\dots=z_{n}=0$. Moreover, the cones are supported in different graded pieces; this is because of our computation in $\C^{n}$ and because the relative Conley-Zehnder indices in a manifold with $c_{1}(TW)=0$ depend only on a neighborhood of the orbits.

  Therefore the continuation morphism:
  \begin{equation*}
    Q(X_{0,1};1;\kappa)\to Q(X_{1,1};1;\kappa)
  \end{equation*}
  vanishes. By similar arguments to the ones encountered in the proof of Lemma \ref{lemma:non-vanishing-result}, we have that:
  \begin{enumerate}
  \item $Q(X_{0,1};1;\kappa)\simeq Q(X_{E};1+\epsilon;\kappa)\simeq Q(E;\kappa)$,
  \item $Q(X_{1,1};1;\kappa)\simeq Q(X_{F};1+\epsilon;\kappa)\simeq Q(F;\kappa)$,
  \end{enumerate}
  via continuation morphisms, and hence $Q(E;\kappa)\to Q(F;\kappa)$ vanishes. This completes the proof of Lemma \ref{lemma:vanishing}.
\end{proof}

\subsection{Proof of the non-squeezing statement}
\label{sec:proof-non-squeezing}
We have proved everything needed to deduce Theorem \ref{theorem:our-non-squeezing} in the case where the squeezings $\psi$ extend to the filling $W$. The argument follows formally from the natural conjugation action.

With a bit more work, one can remove the assumption that the squeezing $\psi$ extends to $W$. The argument then goes as follows: consider a one-parameter family of vector fields $X_{\sigma}$ whose contact Hamiltonian equals:
\begin{equation*}
  h_{\sigma}=1+\frac{\pi \abs{z}^{2}}{(1-s)a+sR}
\end{equation*}
in the neighborhood $\R/\Z\times B(R)$ and which is negative outside the neighborhood. By our assumptions, this family is non-resonant relative $\kappa$. It follows easily that $\psi_{*}X_{\sigma}$ is again a non-resonant family relative the free homotopy class containing $\psi(K)$. Now we invoke the final assumption:
\begin{enumerate}[label=($\ast$)]
\item the knot $K$ is primitive
\end{enumerate}
Since $\psi(K)\subset \R/\Z\times B(R)$, $K$ is homotopic to an iterate of the central knot; but, by the primitive assumption, we must have $\psi(K)\sim K^{\pm 1}$. Think of $\psi$ as defining another contactomorphic embedding of $\R/\Z\times B(R)$ into $\bd W$; if $\psi$ sends $K$ onto $K^{-1}$, then we can simply precompose $\psi|_{\R/\Z\times B(R)}$ with the anticontact involution $(\theta,z)\mapsto (-\theta,\bar{z})$; after precomposing, if necessary, we also ensure that the neighborhood $\psi(\R/\Z\times B(R))$ respects the coorientation.

Our arguments apply equally well to this new embedding $\psi(\R/\Z\times B(R))$, and so we conclude that the continuation morphism:
\begin{equation*}
  Q(\psi(\Omega(a));\kappa)\to Q(\psi(\Omega(R));\kappa).
\end{equation*}
is an isomorphism using the non-resonant deformation as explained above. Moreover, $Q(\psi(\Omega(a));\kappa)$ is non-zero, using the same local Floer homology ideas as in the proof that $Q(\Omega(a);\kappa)$ is non-zero. However, by assumption, $$\psi(\Omega(a))\subset E(c,R,\dots)\subset E(a,R,\dots)\subset \psi(\Omega(R)),$$
and since the continuation morphism:
\begin{equation*}
  Q(E(c,R,\dots);\kappa)\to Q(E(a,R,\dots);\kappa)
\end{equation*}
has been shown to vanish we obtain the desired contradiction, concluding the proof of Theorem \ref{theorem:our-non-squeezing}.

\bibliographystyle{alpha}
\bibliography{citations}
\end{document}